\documentclass{amsart}

\usepackage{lipsum}
\usepackage{amsfonts}
\usepackage{graphicx}
\usepackage{epstopdf}
\usepackage{algorithm}

\usepackage{amsmath,amscd,amssymb}
\usepackage{xfrac}
\usepackage{url}
\usepackage{mathtools}
\usepackage{color}
\usepackage{subcaption}
\usepackage[pagebackref,colorlinks,citecolor=blue,linkcolor=magenta]{hyperref}

\usepackage[capitalize]{cleveref}
\usepackage{tikz}
\usetikzlibrary{calc}
\usepackage{color}
\usepackage{centernot}
\usepackage{enumitem}
\usepackage[utf8]{inputenc}
\usepackage{algpseudocode}

\newcommand{\GG}{\mathcal{G}}
\newcommand{\HH}{\mathcal{H}}
\newcommand{\DD}{\mathcal{D}}
\newcommand{\SSS}{\mathcal{S}}
\newcommand{\RR}{\mathbb{R}}

\renewcommand{\ne}{\operatorname{ne}}

\makeatletter
\DeclareMathAlphabet{\@mymathbb}{U}{bbold}{m}{n}
\newcommand{\zero}{\@mymathbb{0}}
\newcommand{\one}{\@mymathbb{1}}
\makeatother

\DeclareMathOperator{\pa}{pa}
\DeclareMathOperator{\ch}{ch}
\DeclareMathOperator{\de}{de}

\DeclareMathOperator{\nd}{nd}

\DeclareMathOperator{\BIC}{BIC}

\newcommand\independent{\protect\mathpalette{\protect\independenT}{\perp}}
\def\independenT#1#2{\mathrel{\rlap{$#1#2$}\mkern2mu{#1#2}}}

\def\newop#1{\expandafter\def\csname #1\endcsname{\mathop{\rm
#1}\nolimits}}
\newop{STAB}
\newop{conv}
\newop{CIM}
\newop{odd}
\newop{even}

\usepackage{enumitem}
\setlist[enumerate]{leftmargin=.5in}
\setlist[itemize]{leftmargin=.5in}

\newtheorem{theorem}{Theorem}[section]
\newtheorem{proposition}[theorem]{Proposition}
\newtheorem{corollary}[theorem]{Corollary}
\newtheorem{lemma}[theorem]{Lemma}
\newtheorem{conjecture}[theorem]{Conjecture}

\theoremstyle{definition}
\newtheorem{definition}[theorem]{Definition}

\theoremstyle{remark}

\theoremstyle{plain}
\newtheorem{example}[theorem]{Example} 

\title{Greedy Causal Discovery is Geometric}

\author{Svante Linusson
\and Petter Restadh
\and Liam Solus
}

\email[Svante Linusson]{linusson@math.kth.se}
\email[Petter Restadh]{petterre@kth.se}
\email[Liam Solus]{solus@kth.se}

\address{Department of Mathematics\\
    KTH Royal Institute of Technology\\
    SE-100 44 Stockholm, Sweden}

\begin{document}
\maketitle

\begin{abstract}
Finding a directed acyclic graph (DAG) that best encodes the conditional independence statements observable from data is a central question within causality.
Algorithms that greedily transform one candidate DAG into another given a fixed set of moves have been particularly successful, for example the GES, GIES, and MMHC algorithms.
In 2010, Studen\'y, Hemmecke and Lindner introduced the characteristic imset polytope, $\operatorname{CIM}_p$, whose vertices correspond to Markov equivalence classes, as a way of transforming causal discovery into a linear optimization problem.
We show that the moves of the aforementioned algorithms are included within classes of edges of $\operatorname{CIM}_p$ and that restrictions placed on the skeleton of the candidate DAGs correspond to faces of $\operatorname{CIM}_p$.
Thus, we observe that GES, GIES, and MMHC all have geometric realizations as greedy edge-walks along $\operatorname{CIM}_p$.
Furthermore, the identified edges of $\operatorname{CIM}_p$ strictly generalize the moves of these algorithms.
Exploiting this generalization, we introduce a greedy simplex-type algorithm called \emph{greedy CIM}, and a hybrid variant, \emph{skeletal greedy CIM}, that outperforms current competitors among hybrid and constraint-based algorithms.

\end{abstract}

\section{Introduction}
\label{sec: intro}
The use of directed acyclic graphs (DAGs) to model complex systems has increased rapidly during the last thirty years, and today they are used in a wide variety of fields \cite{FLNP00, P00, BHR00, S01}.
Given a positive integer $p$ we let $[p]\coloneqq \{1, 2, \dots, p\}$.
To each DAG $\GG=([p], E)$ we associate a set of random variables $X_1,\dots, X_p$, and the conditional independence (CI) statements $X_i\independent  X_{\nd_\GG(i)\setminus \pa_\GG(i)}|X_{\pa_\GG(i)}$ for all $i\in[p]$.
Here, $\pa_\GG(i)$ denotes the parents and $\nd_\GG(i)$ denotes the non-descendants of $i$ in $\GG$.
A  joint probability distribution $P(X_1,\dots, X_p)$ is \emph{Markov} to a DAG $\GG$ if it entails all such CI statements.
The goal of causal discovery is to learn an unknown DAG $\GG=([p], E)$ from samples drawn from a joint distribution $P$ over $(X_1,\dots,X_p)$ that is assumed to be Markov to $\GG$.
Unfortunately, this cannot generally be done as multiple DAGs can encode the same set of CI statements.
Two such DAGs are called \emph{Markov equivalent}, and they belong to the same \emph{Markov equivalence class} (MEC). 
Thus, the basic problem of causal discovery is to identify the MEC of $\GG$, and a variety of \emph{causal discovery algorithms} for doing so have been proposed \cite{C02, HB12, SG91, T19}. 

Many of the more competitive algorithms are score-based and greedy, like the Greedy Equivalence Search (GES) \cite{C02}, or the Greedy Interventional Equivalence Search (GIES) applied to only observational data \cite{HB12}.
These algorithms aim to maximize a score function, such as the Bayesian Information Criterion (BIC). 
Others aim to recover the MEC from a collection of CI statements by treating causal discovery as a constraint-satisfaction problem, like the PC algorithm \cite{SG91, T19}. 
While the score-based methods tend to be more accurate on both simulated and real data, the constraint-based algorithms are usually faster.
More recent algorithms have tried using a hybrid approach, like Max-Min Hill Climbing (MMHC) \cite{TBA06}, where the authors restrict the search space by using CI tests and then take a greedy score-based approach. 
In the hybrid setting, one can leverage the speed of constraint-based methods versus the accuracy of score-based methods.

Alternatively, Studen\'y, Hemmecke and Lindner gave a geometric interpretation of MECs by realizing them as $0/1$-vectors called \emph{characteristic imsets} \cite{SHL10}. 
Maximizing a score equivalent and (additive) {decomposable} function over the MECs of DAGs on $p$ nodes then becomes equivalent to maximizing a linear function over these vectors.
Thus, finding the BIC-optimal MEC can be seen as a linear optimization problem over the \emph{characteristic imset (CIM) polytope}, $\CIM_p$.
This approach has also been used to learn decomposable models, with promising results \cite{SC17}.
While most research on these polytopes has focused on the identification of facets, our main focus will be their edges and other lower-dimensional faces.

We begin by showing that the reduced search space of the aforementioned popular hybrid and constraint-based algorithms are realized as faces of $\CIM_p$ (see \Cref{prop: graph interval face of cimp}). 
In \Cref{sec: edges}, we then identify classes of edges corresponding to, and strictly generalizing, the moves of GES, GIES, and MMHC.
Thus, we obtain a geometric interpretation of these algorithms as edge-walks along faces of a convex polytope.
A more recent hybrid algorithm called greedy SP \cite{SUW20} also admits a geometric interpretation as an edge-walk along a convex polytope. 
Since GES, GIES, MMHC, and greedy SP are currently the benchmark standards for greedy causal discovery algorithms based solely on observational data, we can then view greedy causal discovery as a purely geometric process; i.e., as an edge-walk along a convex polytope (see \Cref{thm: causal discovery is geometric}). 
Furthermore, as the characterized edges of $\CIM_p$ strictly generalize the moves of GES, GIES, and MMHC, we propose a hybrid algorithm that we call \emph{skeletal greedy CIM}  (\Cref{alg: ske greedy cim}) and a greedy score-based algorithm that we call \emph{greedy CIM} (\Cref{alg: greedy cim}).

In \Cref{sec: results}, we study how greedy CIM and skeletal greedy CIM perform on simulated data and compare their performance with the state-of-the-art. 
We observe that the additional moves given by the classified edges of $\CIM_p$ result in both the hybrid and purely score-based algorithms performing at least as well as all (respective) benchmark standards.  
In the case of hybrid algorithms, skeletal greedy CIM consistently outperforms all other hybrid alternatives.  
These observations purport the edges of the characteristic imset polytope as the natural object of study in efforts to improve the accuracy of modern causal discovery algorithms.  
The more technical proofs of the main theorems in \Cref{sec: edges} can be found in \Cref{sec: proofs edges}.

\section{Preliminaries}
\label{sec: prel}
For an introduction to the theory of convex polytopes, see for example \cite{Zie95}. 
We start with a brief summary of the graph theory notation used in the paper. All graphs are assumed to be simple.

Let $G=([p], E)$ be an undirected graph. 
For a pair of distinct nodes $i,j\in [p]$ we write $i-j\in G$ if $\{i,j\}\in E$.
We denote the set of neighbors of $i$ in $G$ by $\ne_G(i)$. 
For a directed graph $\GG=([p], E)$ we likewise write $i\to k\in \GG$ if $(i, k)\in E$. 
Then $i$ is said to be a \emph{parent} of $k$ and $k$ a \emph{child} of $i$. 
The sets of parents and children of $k$ in $\GG$ are denoted by $\pa_\GG(k)$ and $\ch_\GG(k)$ respectively. 
The \emph{skeleton} of a directed graph $\GG$ is the undirected graph $G$ where we replace $k\to i\in \GG$ with $k-i\in G$.
We say that two nodes are neighbors in $\GG$ if they are neighbors in the skeleton of $\GG$. 
For a directed graph $\GG$ we say that $\langle k_0, k_1,\dots, k_n\rangle$ is a \emph{directed path} from $k_0$ to $k_n$ in $\GG$ if $k_i \to k_{i+1}\in \GG$ for all $0\leq i\leq n-1$
and all $k_i$ different. 
We say that $\langle k_0, k_1,\dots, k_n\rangle$ is a \emph{path} in $\GG$ if it is a path in the skeleton of $G$.
Then a \emph{directed cycle} is a directed path with an extra edge $k_n\to k_0$ and a directed graph $\GG$ is a \emph{directed acyclic graph} (DAG) if $\GG$ does not have a directed cycle.
A node $i$ is a \emph{descendant} of $k$ if there exists a directed path from $k$ to $i$, and $i\neq k$.
The set of descendants is denoted $\de_\GG(k)$, and, by definition, does not include $k$. 
Every node that is not $k$, nor a descendant of $k$, is a \emph{non-descendant}, and the set of all such nodes is denoted $\nd_\GG(k)$.
The \emph{induced subgraph} on $A\subseteq [p]$ is denoted $\GG|_A$.
We recommend \cite{lauritzen1996} for a background on graphs and DAG models. 
\smallskip

A \emph{v-structure} is an induced subgraph of the form $i\to j\leftarrow k$.
The following is a classical result of Verma and Pearl.

\begin{theorem}
\label{thm: verma pearl}
\cite{VP92}
Two DAGs are Markov equivalent if and only if they have the same skeleton and the same v-structures.
\end{theorem}
Let $\GG=([p], E)$ be a DAG.
It is well-known that a joint distribution over $(X_1,\ldots,X_p)$ is Markov to $\GG$ if and only if its probability density function $P$ factorizes as
\begin{equation}
\label{eq: dag factorization}
P(X_1,\dots, X_p) = \prod_{i\in [p]}P(X_i|X_{\pa_\GG(i)}).
\end{equation}
To obtain a unique graphical representation of each MEC, Andersson, Madigan, and Perlman proposed and gave a complete characterization of  \emph{essential graphs}  \cite{AMP97}. 
Studen\'y proposed a more geometric interpretation of Markov equivalence via vectors that encode the CI statements, called the \emph{standard imset} \cite{S04, S05}. 
Following this idea, in \cite{SHL10} Studen\'y, Hemmecke, and Lindner introduced the \emph{characteristic imset}, $c_\GG$, of a DAG $\GG$ that encodes the factorization of \Cref{eq: dag factorization}.
As the factorization determines the MEC, this gives us a unique representation of each MEC. 
Formally it is a function $c_\GG \colon \left\{ S\subseteq [p]\colon |S|\geq 2\right\}\to \{0,1\}$ defined as
\[
c_\GG(S) \coloneqq
\begin{cases}
1	&	\text{ if there exists $i\in S$ such that for all $j\in S\setminus\{i\}$, $j\in\pa_\GG(i)$},	\\
0	&	\text{ otherwise}.	\\
\end{cases}
\]
As $c_\GG$ is a function from a finite set we can identify it with a vector in $\mathbb{R}^{2^p-p-1}$ where the basis vectors, $e_S$, are indexed by the sets in $\left\{ S\subseteq [p]\colon |S|\geq 2\right\}$.
Similar to essential graphs, characteristic imsets then give us a unique representation for each MEC.

\begin{theorem}
\label{lem: studeny}
\cite{SHL10}
Two DAGs $\GG$ and $\HH$ are Markov equivalent if and only if $c_\GG = c_\HH$.
\end{theorem}

The next lemma follows from the definition of characteristic imsets and provides a way to recover the structure of the graph from this vector encoding. 
\begin{lemma}
\label{lem: imset structure}
\cite{SHL10}
Let $\GG$ be a DAG on $[p]$.
Then for any distinct nodes $i$, $j$, and $k$ we have
\begin{enumerate}[label=(\arabic*)]
\item{$i\rightarrow j$ or $i\leftarrow j$ in $\GG$ if and only if $c_\GG(\{i, j\})=1$.}
\item{$i\rightarrow j \leftarrow k$ is a v-structure in $\GG$ if and only of $c_\GG(\{i,j,k\})=1$ and $c_\GG(\{i,k\})=0$.}
\end{enumerate}
\end{lemma}
As we can see, the characteristic imset encodes the skeleton and the v-structures in the 2- and 3-sets. 
Any {(additive) decomposable} and {score equivalent} function can be seen as an affine linear function over the vectors $c_\GG$ \cite{SHL10}. 
An important example of such a function is the \emph{Bayesian Information Criterion} (BIC).
Given $n$ independent samples, $\mathbf{D}$, drawn from the joint distribution of $(X_1, \dots, X_p)$, the BIC is defined as
\begin{equation}
\label{eq: bic}
\BIC(\GG,\mathbf{D}) = \log {P}\left(\mathbf{D}|\hat{\theta}, \GG^h\right)-\frac{d}2\log(n).
\end{equation}
Here $\hat{\theta}$ is the maximum-likelihood estimate for the network parameters,  $d$ denotes the number of free parameters of $\GG$, and $\GG^h$ denotes the hypothesis that $\mathbf{D}$ are i.i.d samples from a distribution that entails exactly the CI statements encoded by $\GG$ \cite{C02}.
Thus, the question of learning the BIC-optimal MEC can be stated as finding the maximum of an affine linear function over a finite set of vectors in a finite-dimensional real vector space.
This motivates the definition of the \emph{characteristic imset polytope} (CIM polytope) for DAGs on $p$ nodes:
\begin{equation*}
\label{eq: cim p}
\CIM_p \coloneqq\conv\left(c_\GG \in \RR^{2^p-p-1} \colon \text{ $\GG = ([p],E)$ a DAG}\right).
\end{equation*}
As $\CIM_p$ is defined as a $0/1$-polytope (that is, the convex hull of vectors with entries that are either $0$ or $1$) the vertices of $\CIM_p$ are precisely $\{c_\GG\colon \text{ $\GG = ([p],E)$ a DAG}\}$ \cite{Zie95}.

A classic constraint-based causal discovery algorithm is the PC algorithm \cite{SG91, T19}. 
It first utilizes CI tests to learn a skeleton $G$ and then to orient v-structures. 
The Max-Min Hill Climbing (MMHC) algorithm \cite{TBA06} utilizes CI tests to learn possible edges in the skeleton, and then uses a score-based method to construct a DAG restricted to using only these edges.
Following the idea of utilizing CI tests to obtain a skeleton, we consider another polytope, closely related to $\CIM_p$.
Let $G = ([p],E)$ be an undirected graph and define the \emph{CIM polytope for $G$} to be 
\begin{equation*}
\label{eq: cim G}
\CIM_G \coloneqq\conv\left(c_\GG \in \RR^{2^p-p-1} \colon\text{ $\GG = ([p],E)$ a DAG with skeleton $G$}\right).
\end{equation*}
Thus, like the PC algorithm, we can learn an undirected skeleton, $G$, via CI tests, and then take a score-based approach via an edge-walk on $\CIM_G$ optimizing the BIC.
Such a method is called a \emph{hybrid algorithm} as it first uses a {constraint-based} approach to restrict the search space, and then uses a {score-based} approach to find the optimal DAG.
%
%
An immediate question is then: what is the relationship between $\CIM_p$ and $\CIM_G$?
To this end we have the following proposition:
\begin{proposition}\label{prop: graph interval face of cimp}
Let $H=([p],E)$ and $H'=([p],E')$ be two undirected graphs such that $E\subseteq E'$.
Then 
\[
\conv\left(c_\GG\in\mathbb{R}^{2^p-p-1}\colon \GG \text{ a DAG with skeleton $G=([p], D)$ where $E\subseteq D\subseteq E'$}\right)
\]
is a face of $\CIM_p$. 
\end{proposition}

\begin{proof}
 It is enough to find a cost function, $w_{H, H'}$, which maximizes precisely over the set
\[
\left\{c_\GG\in\mathbb{R}^{2^p-p-1}\colon \GG \text{ a DAG with skeleton $G=([p], D)$ where $E\subseteq D\subseteq E'$}\right\},
\]
out of all characteristic imsets.
So define
\[
w_{H, H'}(S)\coloneqq
\begin{cases}
0 & \text{if $|S|\neq 2$ or $S\in E'\setminus E$,}\\
1 & \text{if $S\in E$,}\\
-1 & \text{otherwise.}
\end{cases}
\]
Notice that $w_{H, H'}$ is only non-zero for sets with cardinality $2$.
Then if we have $\GG$ with skeleton $G=([p], D)$ we get via \Cref{lem: imset structure}
\begin{align*}
w^T_{H,H'}c_\GG&=\sum_{S\subseteq [p], |S|\geq2}w_{H,H'}(S)c_\GG(S)\\
&=|D\cap E|-|D\setminus E'|.
\end{align*}
The right-hand-side is maximized exactly when $E\subseteq D\subseteq E'$.
Thus $w_{H, H'}$ maximizes exactly over the given set.
\end{proof}
Taking $H=H'=G$, we get the following corollary.
\begin{corollary}\label{prop: graph face of cimp}
Let $G=([p],E)$ be an undirected graph. 
Then $\CIM_G$ is a face of $\CIM_p$. 
\end{corollary}

\section{Edges of the $\CIM$ Polytope}
\label{sec: edges}
To construct efficient algorithms for finding the maximum of a linear score function over a polytope, we need some description of the polytope. 
Assume we are given an arbitrary polytope $Q$ and a linear function $s$. 
It is immediate that the set maximizing $s^Tq$ for $q\in Q$ is a face of $Q$. 
Thus, any linear function assumes its maximum value over $Q$ at at least one vertex of $Q$. 
An \emph{edge-walk} on $Q$ to maximize a linear function $s$ is done in the following way: 
Start at any vertex $q_0$ of $q$, and set $i\coloneqq 0$.  
At each step, choose $q_{i+1}$ such that  $\conv(q_i, q_{i+1})$ is an edge of the polytope and $s^Tq_i < s^Tq_{i+1}$.
If no such edges exist, return $q_i$. 

Assuming that we know every edge of $Q$, such an edge-walk will always return a vertex maximizing $s$. 
Making additional assumptions on the score function or looking for edges in a certain order can sometimes give us similar guarantees. 
For example, see \cite{SUW20}.
As a direct computation of all edges of $\CIM_p$ and $\CIM_G$ is not feasible for large $p$, we will instead identify edges of these polytopes in terms of relations between the characteristic imsets they connect. 
As the characteristic imsets are the vertices of these polytopes we will see how these relations label edges of the polytope.

We will define two relations, one on $\CIM_G$ and one on $\CIM_p$. 
Utilizing the first one we propose a hybrid algorithm that first learns an undirected skeleton via conditional independence tests and then performs an edge-walk along $\CIM_G$, greedily optimizing $\BIC$. 
Then using both we also define a purely score-based algorithm that performs an edge-walk on $\CIM_p$, again greedily optimizing the $\BIC$. 

The edges we identify will include, as a special case, the moves of Greedy Equivalence Search (GES) \cite{C02}. 
This positively answers a question raised by Steffen Lauritzen at the \emph{Workshop on Graphical Models: Conditional Independence and Algebraic Structures, TU Munich, 2019}: Do the moves of GES have a geometric interpretation in terms of the $\CIM_p$ polytope?
More generally, we recover a geometric interpretation of the GIES algorithm \cite{HB12} in the case of purely observational data, as well as the hybrid MMHC algorithm.

To this end, we will begin by defining relations between imsets. 
These relations are motivated by our graphical understanding of Markov equivalence, but also turn out to generalize our intuition. 
%

\begin{definition}
[Turn pair]
\label{def: turn pair}
Let $\GG$ and $\HH$ be two DAGs on node set $[p]$ and with skeleton $G$.  
Suppose there exist $i$,  $j$, $S_i\subseteq[p]\backslash \{i,j\}$ and $S_j\subseteq[p]\backslash \{i,j\}$ such that 
\begin{enumerate}[label=(\arabic*)]
\item{$c_\GG(\{i,j\})=1$;}
\item{$c_\GG(S\cup \{i\})=1$ for all $S\subseteq S_i$ with $|S|\geq 1$;}
\item{$c_\GG(S\cup \{j\})=1$ for all $S\subseteq S_j$ with $|S|\geq 1$;}
\item{either $S_i\not\subseteq\ne_G(j)$ or $S_j\not\subseteq\ne_G(i)$.}
\end{enumerate}
Then we say that $\{\GG, \HH\}$ is a \emph{turn pair} with respect to $(i,j,S_i, S_j)$ if 
\[
c_{\HH} = c_\GG +\sum_{S\in \SSS^+}e_S - \sum_{S\in \SSS^-}e_S
\]
where  $\SSS^+\coloneqq \{T\cup \{i,j\}\colon T\subseteq S_i,\  T \not\subseteq\ne_G(j)\}$ and $\SSS^-\coloneqq \{T\cup \{i,j\}\colon T\subseteq S_j,\ T \not\subseteq\ne_G(i)\}$.
\end{definition}
Note that one of $S_i$ and $S_j$ can may be empty, but not both by (4). 
The name ``turn pair'' is explained via the next proposition.
We observe that $\{\GG, \HH\}$ is a turn pair with respect to $(i,j,S_i,S_j)$ if and only if $\{\HH, \GG\}$ is a turn pair with respect to $(j,i,S_j,S_i)$. 
Moreover, as the edges of a polytope lack orientation, a greedy edge-walk may walk in either direction along a given edge. 
Thus we view our relations between characteristic imsets and their corresponding DAGs as unordered pairs, as opposed to ordered pairs.

If $\GG$ is a directed graph with $i\to j\in\GG$ we denote by $\GG_{i\leftarrow j}$ the directed graph identical to $\GG$ except that the edge $i\to j$ is replaced with $i\leftarrow j$. 

\begin{proposition}
\label{prop: edge turn}
Let $\GG$ be a DAG with $i\to j\in\GG$.
If $\GG_{i\leftarrow j}$ is a DAG, then either $\GG$ and $\GG_{i\leftarrow j}$ are Markov equivalent, or $\{\GG, \GG_{i\leftarrow j}\}$ is a turn pair.
\end{proposition}
The case in which $\GG$ and $\GG_{i\leftarrow j}$ are Markov equivalent is characterized in \cite{C95}. 
Hauser and B\"uhlmann define a collection of turning moves in terms of the essential graph (see \cite[Propositions~$31$ and~$34$]{HB12}). 
They characterize when $\GG_{i\leftarrow j}$ is a DAG and the relation between the essential graphs of $\GG$ and $\GG_{i\leftarrow j}$ when this is the case.
The above proposition shows that in the case of no interventions, their turning moves are turn pairs. 
The converse is not true, as shown in \Cref{ex: bud not flip of edge}. 

\begin{example}
\label{ex: bud not flip of edge}
By \Cref{prop: edge turn}, turn pairs capture whenever we turn an edge in a DAG, transforming it into another (non-Markov equivalent) DAG, in terms of characteristic imsets.
The converse of \Cref{prop: edge turn} is, on the other hand, not true. 
That is, there exists a turn pair $\{\GG, \HH\}$ for which there is no DAG $\DD$ Markov equivalent to $\GG$ such that $\DD_{i\leftarrow j}$ is Markov equivalent to $\HH$. 
As an example of this, take $\GG$ and $\HH$ as in \Cref{fig: ex bud not flip of edge}.
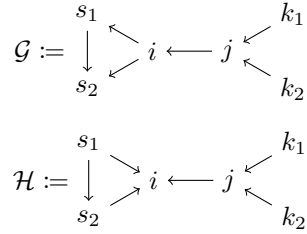
\begin{figure}
\[
\begin{tikzpicture}
\node at (-1.5,0){$\GG\coloneqq$};
\node (s1) at (150:1) {$s_1$};
\node (s2) at (-150:1) {$s_2$};
\node (i) at (0,0) {$i$};
\node (j) at (1,0) {$j$};
\node (k1) at ($(1,0)+(30:1)$) {$k_1$};
\node (k2) at ($(1,0)+(-30:1)$) {$k_2$};

\foreach \from/\to in {s1/s2, i/s1, i/s2, j/i, k1/j, k2/j}{
	\draw[->] (\from) -- (\to);
}

\end{tikzpicture}
\]
\[
\begin{tikzpicture}
\node at (-1.5,0){$\HH\coloneqq$};
\node (s1) at (150:1) {$s_1$};
\node (s2) at (-150:1) {$s_2$};
\node (i) at (0,0) {$i$};
\node (j) at (1,0) {$j$};
\node (k1) at ($(1,0)+(30:1)$) {$k_1$};
\node (k2) at ($(1,0)+(-30:1)$) {$k_2$};

\foreach \from/\to in {s1/s2, s1/i, s2/i, j/i, k1/j, k2/j}{
	\draw[->] (\from) -- (\to);
}

\end{tikzpicture}
\]
\caption{An example of a turn pair not arising from changing the direction of any one edge in any DAG in the MEC.}
\label{fig: ex bud not flip of edge}
\end{figure}
It can be checked that $\{\GG,\HH\}$ is a turn pair with respect to $(i,j,\{s_1,s_2\}, \emptyset)$, but it follows from \Cref{thm: verma pearl} that $i\leftarrow j\in\DD$ for all DAGs $\DD$ Markov equivalent to $\GG$ or $\HH$.
\end{example}

By \Cref{prop: edge turn}, turn pairs arise naturally from an intuitive graphical interpretation of reversing an edge and, as \Cref{ex: bud not flip of edge} shows, strictly generalize this intuition.
\begin{theorem}
\label{thm: turn pair}
If $\{\GG, \HH\}$ is a turn pair, then $\conv(c_\GG, c_\HH)$ is an edge of $\CIM_G$ where $G$ is the skeleton of $\GG$ and $\HH$.
\end{theorem}

The above theorem tells us that moving via turn pairs is in fact an edge-walk along $\CIM_G$.
For algorithms based on such edge-walks to be able to perform well we would like to move around $\CIM_G$ relatively freely. 
In the following proposition we show that the edges labeled by turn pairs are enough to traverse the polytope $\CIM_G$. 

\begin{proposition}
\label{prop: traversing}
Let $G$ be a graph and $\GG$ and $\HH$ two DAGs with skeleton $G$. Then there exists a sequence of edges 
$
\conv(c_\GG,c_{\DD_1}), \conv(c_{\DD_1},c_{\DD_2}), \ldots, \conv(c_{\DD_{m-1}},c_{\DD_m}), \conv(c_{\DD_m},c_{\HH}),
$ 
of $\CIM_G$ such that each pair $\{\GG,\DD_1\}, \{\DD_1,\DD_2\}, \ldots, \{\DD_{m-1},\DD_m\}, \text{ and } \{\DD_m,\HH\}$ is a turn pair.  
\end{proposition}

\begin{proof}
By \cref{prop: edge turn} it is enough to show that there exists a sequence of DAGs $\GG=\DD_0, \dots, \DD_n=\HH$ such that $\DD_i$ and $\DD_{i+1}$ differ by the direction of a single edge. 
To find such a sequence it is enough to show that for any two DAGs, $\GG$ and $\HH$, that share the same skeleton, there exists an edge $i- j\in G$ such that $i\to j\in\GG$, $i\gets j\in\HH$, and $\GG_{i\gets j}$ is a DAG. 

We can partially order all edges via $i'\to j'\preceq i\to j$ if and only if $j'\in\de_\GG(j)$ or, if $j'=j$, $i\in\de(i')$. 
Note that we sort the children according to $\GG$ and the parents in reverse.
Consider all edges that differ between $\GG$ and $\HH$ and consider such an edge $i\to j$
that is maximal in the prescribed order.
For the sake of contradiction assume there is a cycle in $\GG_{i\gets j}$.
Then there is a directed path $i\to \dots \to j$ different from the edge $i\to j$. 
However, every edge in this path is bigger in the order $\preceq$, and hence this path is present in $\HH$ as well. 
This gives us a directed cycle in $\HH$, a contradiction.
Hence, with this choice of the edge $i\to j$, $\GG_{i\gets j}$ will be a DAG and the result follows. 
\end{proof}

\begin{algorithm}[t!]
\caption{Skeletal Greedy CIM}
\label{alg: ske greedy cim}
\textbf{Input:}{ Data $\mathbf{D}$.}\\
\textbf{Output:}{ A characteristic imset $c_\GG$.}
\begin{algorithmic}
\State{Perform CI tests to find the underlying skeleton $G$\footnotemark}
\State{Let $\GG$ a DAG with skeleton $G$}
\State{$c_\DD\gets \texttt{null}$}
\While{$c_\DD \neq c_\GG$}
\State{$c_\GG\gets c_\DD$}
\State{$c_\DD\gets$ turn phase (\Cref{alg: turn phase} in \Cref{sec: algorithms} with $c_\DD$ as input.)}
\EndWhile
\Return{$c_\GG$}
\end{algorithmic}
\end{algorithm}
\footnotetext{For example we can use the \texttt{skeleton} algorithm from the \texttt{pcalg} package in R \cite{HB12, KMMCMB12}.}
The edges labeled by turn pairs thus connect the polytope $\CIM_G$ in the sense that for any two DAGs, $\GG$ and $\HH$, with skeleton $G$, there exists a sequence of turn pairs that begins at $c_\GG$ and ends at $c_\HH$. 
Thus we can take a simplex-type approach to finding the BIC-optimal MEC.
To this end, we propose a hybrid greedy causal discovery algorithm in which we first learn the skeleton $G$ via CI tests, similar to the PC algorithm, and then perform a restricted edge-walk on $\CIM_G$ utilizing the edges labeled by turn pairs, which we call the \emph{turn phase}.
We call this algorithm \emph{skeletal greedy CIM} (see \Cref{alg: ske greedy cim}).

Up until now we have primarily studied $\CIM_G$, but we would like to move between vertices of $\CIM_G$ and $\CIM_H$ when $G$ and $H$ are not equal.
A direct consequence of \Cref{prop: edge turn} and \Cref{thm: turn pair} is that the turning phase of GIES \cite{HB12} is a type of edge-walk over $\CIM_G$. 
The question then arises whether it holds for the forward and backward phases as well.
Thus, we would like a definition similar to \Cref{def: turn pair} but for adding an edge.
\begin{definition}
[Edge pair]
\label{def: edge pair}
Let $\GG$ and $\HH$ be two DAGs on node set $[p]$.  
Suppose there exists distinct nodes $i$, $j$ and a set $S^\ast\subseteq[p]\backslash \{i,j\}$ such that 
\begin{enumerate}[label=(\arabic*)]
\item{$c_\GG(\{i,j\})=0$,}
\item{$c_\GG(S\cup \{i\})=1$ for all $S\subseteq S^\ast$ with $|S|\geq 1$.}
\end{enumerate}
Then we say that $\{\GG, \HH\}$ is an \emph{edge pair} with respect to $(i,j,S^\ast)$ if
\[
c_{\HH} = c_\GG +\sum_{S\in \SSS_{+i\leftarrow j}}e_S
\]
where  $\SSS_{+i\leftarrow j}\coloneqq \{S\cup \{i,j\}\colon S\subseteq S^\ast\}$.
\end{definition}

Let $\GG$ be a DAG and assume $i$ and $j$ are not adjacent in the skeleton of $\GG$. 
We denote by $\GG_{+i\leftarrow j}$ the directed graph identical to $\GG$ with the edge $i\leftarrow j\in\GG_{+i\leftarrow j}$.
Then, similar to \Cref{prop: edge turn}, we have the following:
\begin{proposition}
\label{prop: edge addition}
Let $\GG$ be a DAG and assume $i$ and $j$ are not adjacent in the skeleton of $\GG$. 
If $\GG_{+i\leftarrow j}$ is a DAG, then $\{\GG,\GG_{+i\leftarrow j}\}$ is an edge pair.
\end{proposition}
Thus edge pairs give an interpretation, in terms of characteristic imsets, of adding an edge to a graph the same way as turn pairs give an interpretation of changing the direction of an edge.
However, in this case we believe that the converse holds. 
\begin{conjecture}
\label{con: edge pair characterization}
Let $\{\GG, \HH\}$ be an edge pair with respect to $(i,j,S^\ast)$. 
Then there exists a DAG $\GG'$ Markov equivalent to $\GG$ such that $\GG'_{+i\leftarrow j}$ is a DAG Markov equivalent to $\HH$. 
\end{conjecture}
Similar to turn pairs, edge pairs constitute edges of $\CIM_p$.
\begin{theorem}
\label{thm: edge pair}
If $\{\GG, \HH\}$ is an edge pair, then $\conv(c_\GG, c_\HH)$ is an edge of $\CIM_p$ where $p$ is the number of nodes in $\GG$ and $\HH$.
\end{theorem}

\begin{algorithm}[t!]
\caption{Greedy CIM}
\label{alg: greedy cim}
\raggedright
\textbf{Input:}{ Data $\mathbf{D}$.}\\
\textbf{Output:}{ A characteristic imset $c_\GG$.}
\begin{algorithmic}
\State{Let $\GG$ be the DAG without any edges}
\State{$c_\DD\leftarrow \texttt{null}$}
\While{$c_\DD \neq c_\GG$}
\State{$c_\GG\gets c_\DD$}
\State{$c_\DD\gets$ edge phase (\Cref{alg: edge phase} in \Cref{sec: algorithms} with $c_\DD$ as input.)}
\State{$c_\DD\gets$ turn phase (\Cref{alg: turn phase} in \Cref{sec: algorithms} with $c_\DD$ as input.)}
\EndWhile\\
\Return{$c_\GG$}
\end{algorithmic}
\end{algorithm}
By combining \Cref{prop: edge addition} with \Cref{thm: edge pair} we obtain a positive answer to the aforementioned question by Steffen Lauritzen; namely, we see that the moves of GES have a geometric interpretation as edges of $\CIM_p$. 
Going even further, by combining this observation with \Cref{prop: edge turn} and \Cref{thm: turn pair}, we see that the moves of the GIES algorithm, which (in the case of purely observational data) extends GES with an additional turn phase, also admit a geometric interpretation as edges of $\CIM_p$.
Similarly, the MMHC algorithm performs a greedy search akin to that of GES, but it first restricts the search space to a subset of edges that are allowed to appear in the skeleton.
An application of \Cref{prop: graph interval face of cimp} with $H=([p], \emptyset)$ thus extends these results to the MMHC algorithm. 
Since greedy SP \cite{SUW20} is defined as an edge-walk along another family of convex polytopes (called DAG associahedra \cite{MUW18}), these observations imply that the popular greedy score-based and hybrid causal discovery algorithms (GES, GIES, MMHC, and greedy SP) can all be viewed as edge-walks along a convex polytope. 
Thus greedy causal discovery is, in a sense, geometric.
We summarize this observation in the following theorem:
\begin{theorem}
\label{thm: causal discovery is geometric}
The following causal discovery algorithms are greedy edge-walks along a convex polytope:
\begin{enumerate}
\item GES,
\item GIES with purely observational data,
\item MMHC, and
\item Greedy SP.
\end{enumerate}
\end{theorem}
\Cref{ex: bud not flip of edge} further shows that the edges of $\CIM_p$ labeled by turn and edge pairs are a strict generalization of the moves of GES and GIES. 
Hence, any edge-walk that greedily optimizes BIC over $\CIM_p$ can be viewed as an extension of these causal discovery algorithms.

In regards to \Cref{thm: edge pair}, we propose the purely score-based algorithm \emph{greedy CIM} (\Cref{alg: greedy cim}) which extends GES and GIES.
This algorithm is, as opposed to skeletal greedy CIM (\Cref{alg: ske greedy cim}), not a hybrid algorithm, as we do not rely on conditional independence tests to find the skeleton. 
Instead it relies on an \emph{edge phase} that consists of a restricted edge-walk, utilizing the edges of $\CIM_p$ determined by edge pairs.
Due to \Cref{thm: edge pair}, the greedy CIM algorithm consists solely of an edge-walk on $\CIM_p$. 
In \Cref{sec: results} we analyze how greedy CIM and skeletal greedy CIM perform on simulated data relative to GES, GIES, MMHC, greedy SP, and the PC algorithm. 

\section{Simulations}
\label{sec: results}
In \Cref{sec: edges} we proposed two algorithms, skeletal greedy CIM (\Cref{alg: ske greedy cim}), and greedy CIM (\Cref{alg: greedy cim}). %
Here we compare the performance of these algorithms on simulated data with the state-of-the-art.

An implementation of all algorithms discussed in this section is available at \cite{github}. 
The simulated data was produced in $R$ \cite{R} using linear structural equation models with Gaussian noise.
The true underlying DAG $\GG^\ast$ was chosen randomly using an Erd\H{o}s-R\'enyi model on $p=8$ vertices and expected neighborhood size $d$, which we varied over the interval $[0.5, 7]$.
Each edge $i\to j$ was given an edge-weight $w_{i,j}$ chosen uniformly from $[-1, -0.25]\cup [0.25, 1]$.
The direction of the edges were given by a linear order of the vertices, sampled uniformly from all linear orders.
We then sampled from a multivariate Gaussian distribution over the random variables $X_1,\dots,X_p$ where $X_i=\varepsilon_i+\sum_{k\in\pa_{\GG^\ast}(i)}w_{k,i}X_k$. 
Here, the $\varepsilon_i$ are independent and normally distributed random variables with mean $0$ and variance $1$.
We produced $100$ models for each $d$ and from each model we drew $n=10,000$ samples.
This was done via the \texttt{MASS} library \cite{MASS}.
As the implementation of greedy CIM and skeletal greedy CIM available at \cite{github} is done in Python, we used the \texttt{rpy2} module for the $R$-to-Python conversions.

To produce the undirected skeleton  in \Cref{alg: ske greedy cim} we used the \texttt{skeleton} algorithm in the \texttt{pcalg} package \cite{HB12, KMMCMB12}.
The algorithm \texttt{skeleton} requires a significance level $\alpha$ for the CI tests, which we varied over $\{0.01, 0.001, 0.0001\}$.
Skeletal greedy CIM, greedy CIM, GES, and GIES all aim to optimize BIC (see \Cref{eq: bic}) which we computed for our models via the \texttt{GaussL0penObsScore}-class from the \texttt{pcalg} package.
In order to fairly compare the different algorithms that are each a single edge-walk along a convex polytope (according to \Cref{thm: causal discovery is geometric}) we ran greedy SP with no restarts and unbounded search depth ($r=1$, $d=\infty$) \cite{SUW20}. 
(Note this choice of parameter settings results in greedy SP performing worse than it did for the same simulations in \cite[Figure 5]{SUW20}, as the parameter settings used to generate \cite[Figure 5]{SUW20} were $r = 10$ and $d = 4$.)  
In \Cref{fig: algo rec} we see the ratio of models recovered from the samples versus the expected neighborhood size $d$.
In \Cref{fig: algo shd} we compare the model recovery rate and the average  Structural Hamming Distance (SHD) (see \cite{TBA06} for a definition) to the true model versus the average expected neighborhood size $d$.

\begin{figure}
\begin{subfigure}{0.45\textwidth}
\includegraphics[width = 5.5cm]{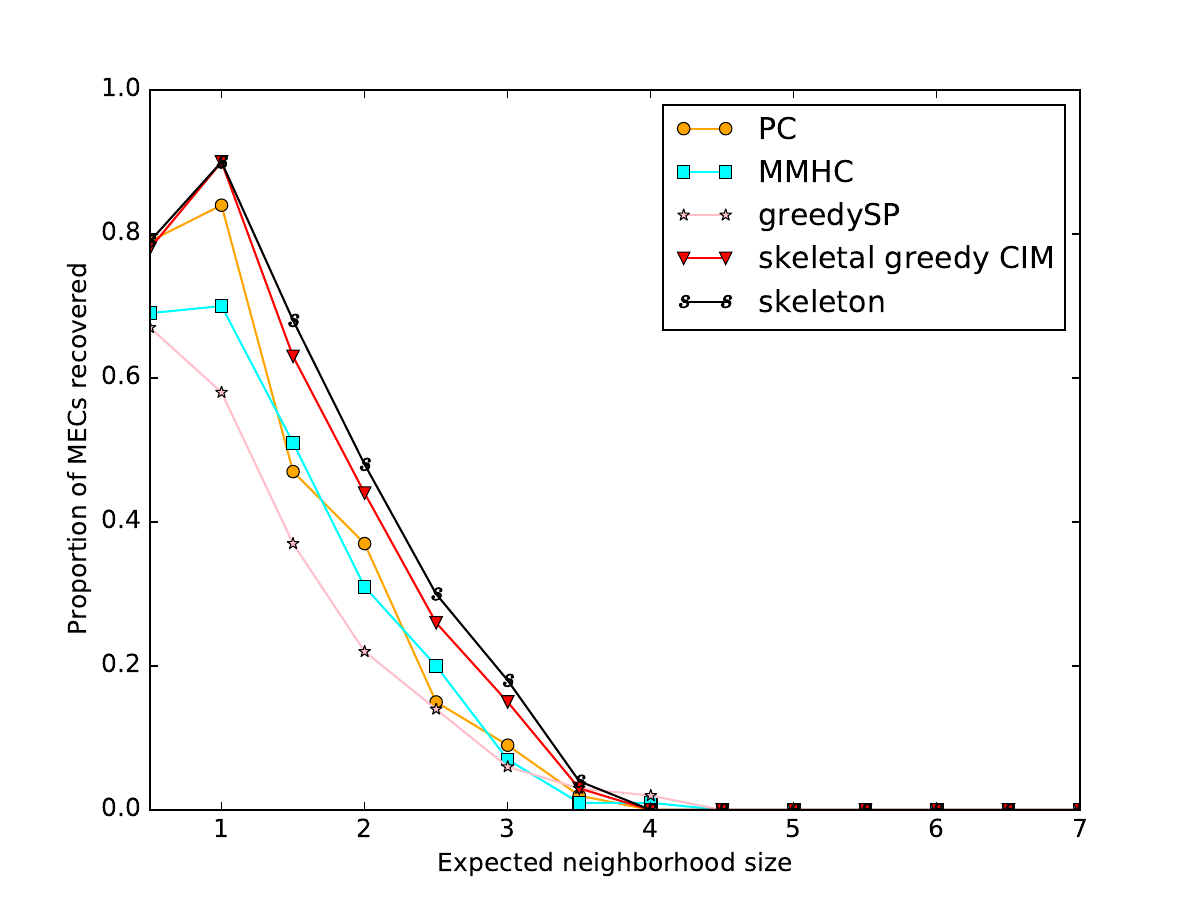}
\caption{$p=8$, $n=10,000$, $\alpha = 0.01$}
\label{subfig: algo rec pc ske mmhc 0.01}
\end{subfigure}
\hfill
\begin{subfigure}{0.45\textwidth}
\includegraphics[width = 5.5cm]{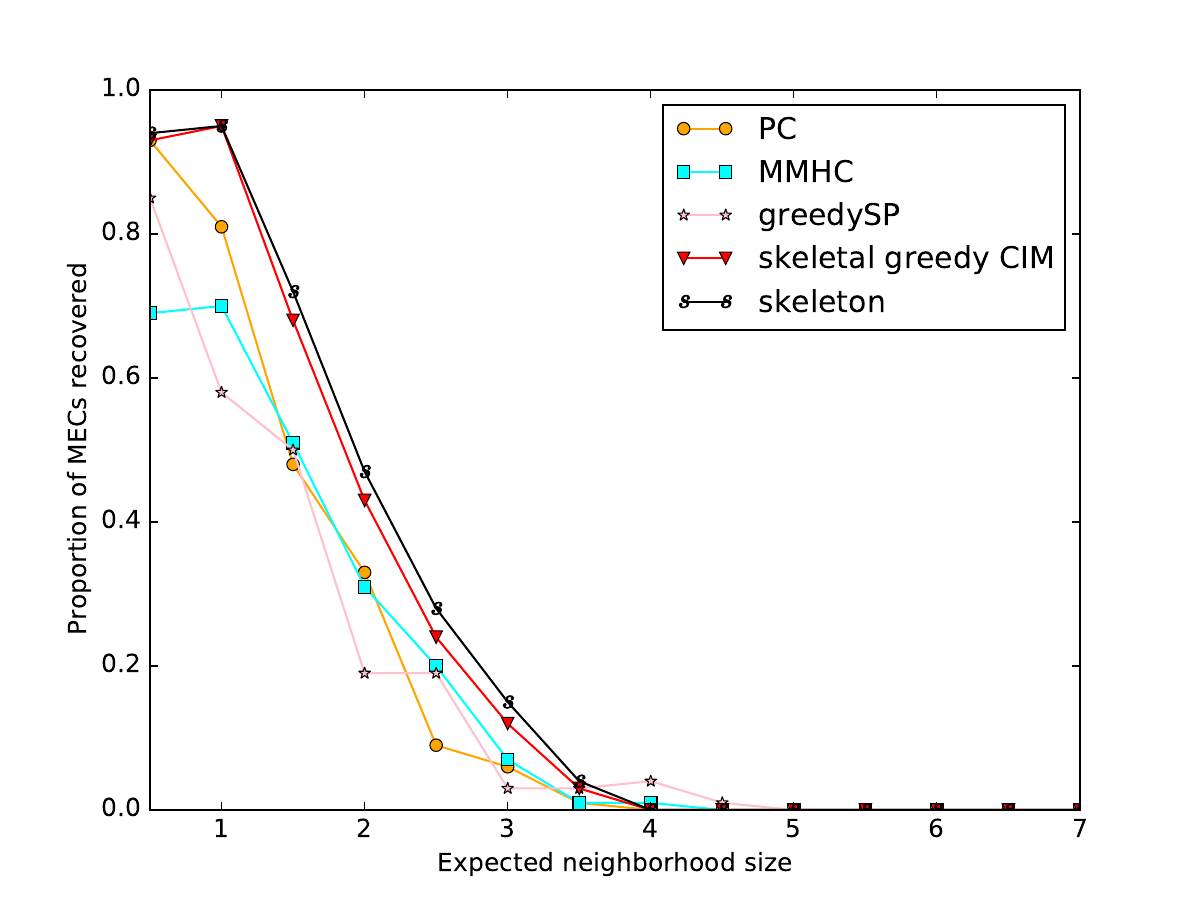}
\caption{$p=8$, $n=10,000$, $\alpha = 0.001$}
\label{subfig: algo rec pc ske mmhc 0.001}
\end{subfigure}
\newline
\begin{subfigure}{0.45\textwidth}
\includegraphics[width = 5.5cm]{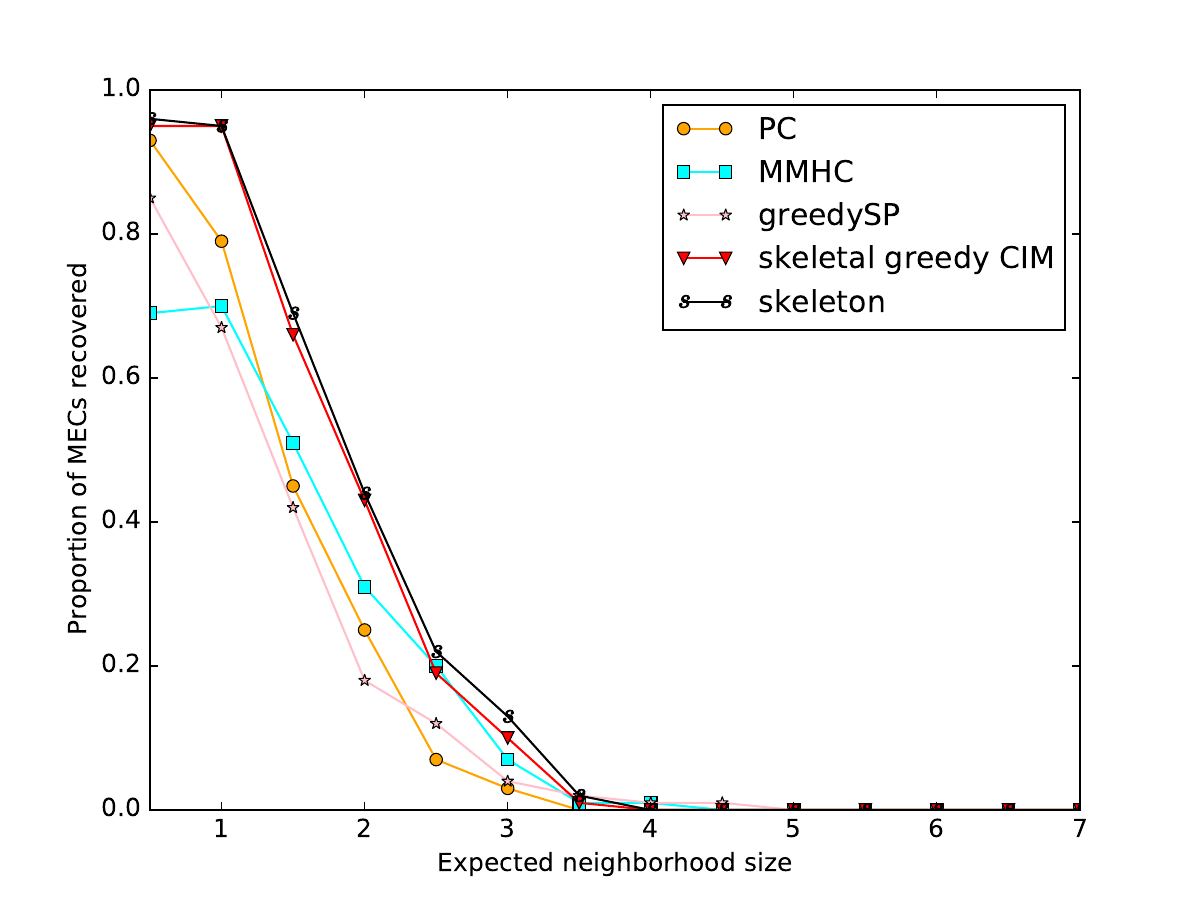}
\caption{$p=8$, $n=10,000$, $\alpha = 0.0001$}
\label{subfig: algo rec pc ske mmhc 0.0001}
\end{subfigure}
\hfill
\begin{subfigure}{0.45\textwidth}
\includegraphics[width = 5.5cm]{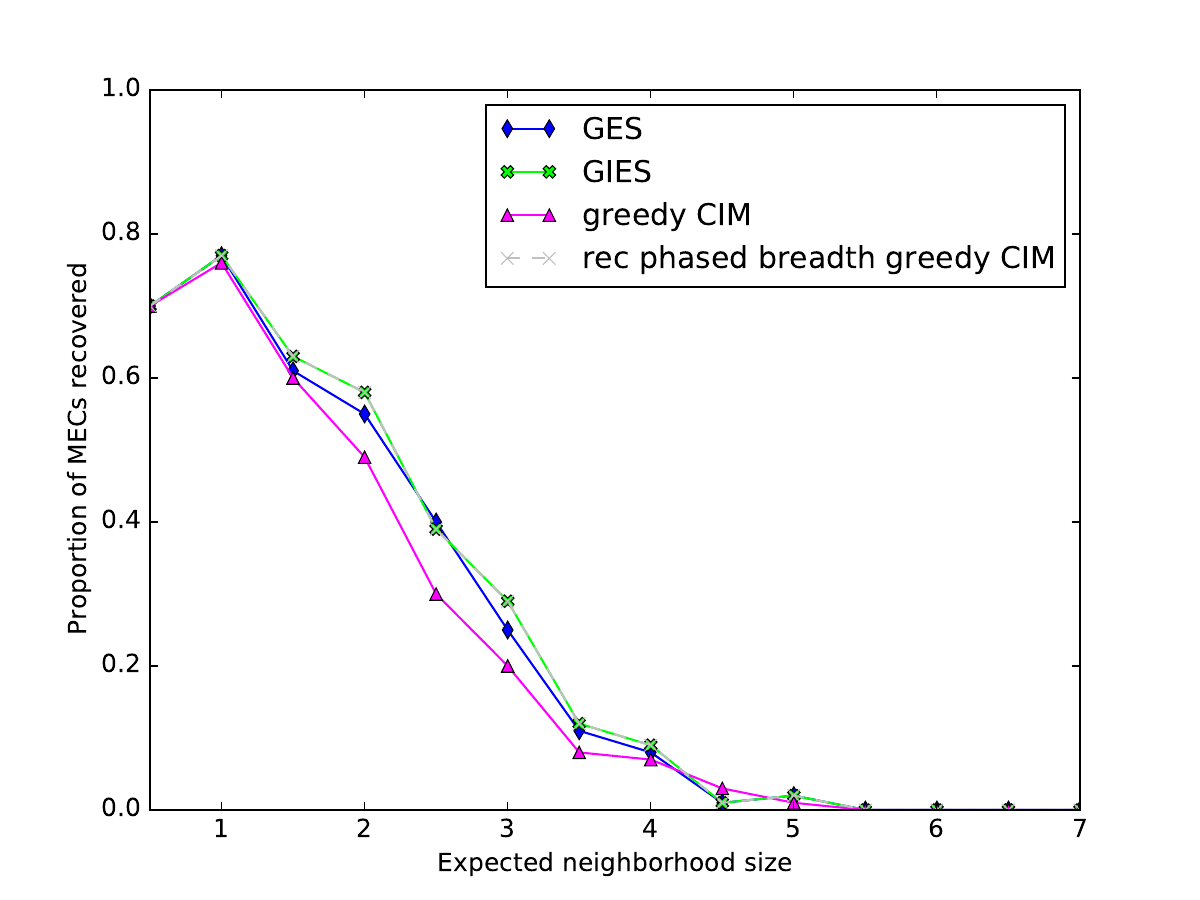}
\caption{$p=8$, $n=10,000$}
\label{subfig: algo rec ges gcim}
\end{subfigure}
\caption{Ratio of models recovered versus the expected neighborhood size of the true graph. In \cref{subfig: algo rec pc ske mmhc 0.01}--\cref{subfig: algo rec pc ske mmhc 0.0001} we ran PC, MMHC, and skeletal greedy CIM on $100$ models. Each model had $p=8$ nodes and the weights of the edges were sampled uniformly in $[-1, -0.25]\cup [0.25, 1]$. We used a sample size of $n=10, 000$, and varied $\alpha$ in $\{0.01, 0.001, 0.0001\}$. In \Cref{subfig: algo rec ges gcim} we see how GES, GIES, and greedy CIM perfomed on the same data. }
\label{fig: algo rec}
\end{figure}

In \Cref{subfig: algo rec pc ske mmhc 0.01}, \Cref{subfig: algo rec pc ske mmhc 0.001}, and \Cref{subfig: algo rec pc ske mmhc 0.0001} we compared all algorithms relying on CI tests (i.e., all constraint-based and hybrid algorithms). 
We see that skeletal greedy CIM has a higher recovery rate than greedy SP, MMHC and the PC algorithm. 
Note that both skeletal greedy CIM and PC are restricted by the performance of the \texttt{skeleton} algorithm, which is the algorithm used to identify the skeleton of the learned DAG.   
Thus, we have also included how often \texttt{skeleton} finds the true skeleton. 
We see that, if the correct skeleton is identified, skeletal greedy CIM almost always learns the true MEC.
However, the same is not true for the PC algorithm.
Based on this near optimality of skeletal greedy CIM, we cannot expect the performance of skeletal greedy CIM to increase by much, even if more edges of $\CIM_G$ are identified and added to the implementation.
The main difference of skeletal greedy CIM  and MMHC is that skeletal greedy CIM relies on CI tests to determine the skeleton, as opposed to MMHC, which only restricts to a set of possible skeletons.
The fact that skeletal greedy CIM outperforms MMHC in \Cref{fig: algo rec} suggests that the set of moves used by skeletal greedy CIM, given by turn pairs, is diverse enough that there is no advantage of hybrid methods that rely on score-based edge specification from a restricted set compared to methods that fully specify a skeleton and then rely on turning edges.
Computational results regarding $\CIM_4$ also suggest that the number of turn pairs make up for a significant part of the edges of $\CIM_G$, but edge pairs make up for a small part of edges of $\CIM_p$ (less than a quarter for $p=4$). 
Thus MMHC might be rather restricted when moving between MECs with different skeletons, which is a non-issue for skeletal greedy CIM.

In \Cref{subfig: algo rec ges gcim} we compared the purely score-based algorithms.
By \Cref{prop: edge turn} and \Cref{prop: edge addition}, greedy CIM can do all moves of GES and GIES, and more.
Recall that greedy CIM was implemented using edge pairs and turn pairs, performing only a depth-first search, whereas GES and GIES perform recurrent phased, breadth-first searches.
To estimate the extent to which turn pairs and edge pairs generalize the moves of GES and GIES, we also implemented a recurrent phased breadth-first version of greedy CIM.
That is, we first only consider edge pairs that increase the number of edges, then the ones that decrease the number of edges, then we enter the turn phase.
We then cycle through these three phases, analogous to GIES.
We call this algorithm \emph{recurrent phased breadth-first greedy CIM}.
As can be seen in \Cref{subfig: algo rec ges gcim}, this version of greedy CIM replicates the output of GIES. 
On the other hand, GES and GIES perform better than greedy CIM. 
This suggests that recurrent phased approaches to optimizing BIC will typically yield better results.
Moreover, the fact that recurrent phased breadth-first greedy CIM matches the best performing algorithm (GIES) suggests that characterizing more edges of $\CIM_p$ and incorporating them into the implementation of greedy CIM could yield even better performing greedy score-based causal discovery algorithms.
The previously mentioned computational results for $\CIM_4$ suggest that there is much room for improvement in this direction as the turn pairs and edge pairs make up less than a quarter of the edges for $\CIM_4$. 

\begin{figure}
\begin{subfigure}[t]{0.45\textwidth}
\includegraphics[width = 5.5cm]{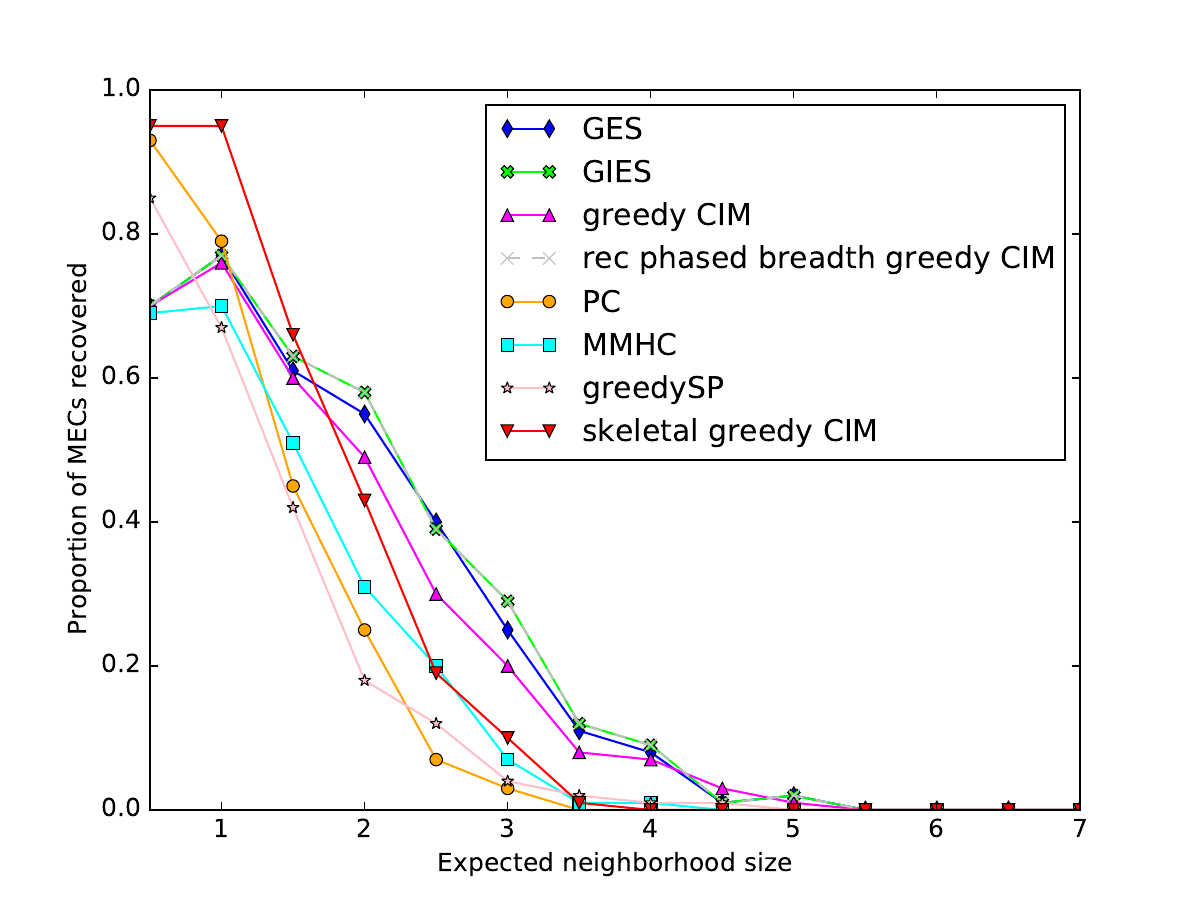}
\caption{The ratio of models recovered.}
\label{subfig: algo rec all 0.0001}
\end{subfigure}
\hfill
\begin{subfigure}[t]{0.45\textwidth}
\includegraphics[width = 5.5cm]{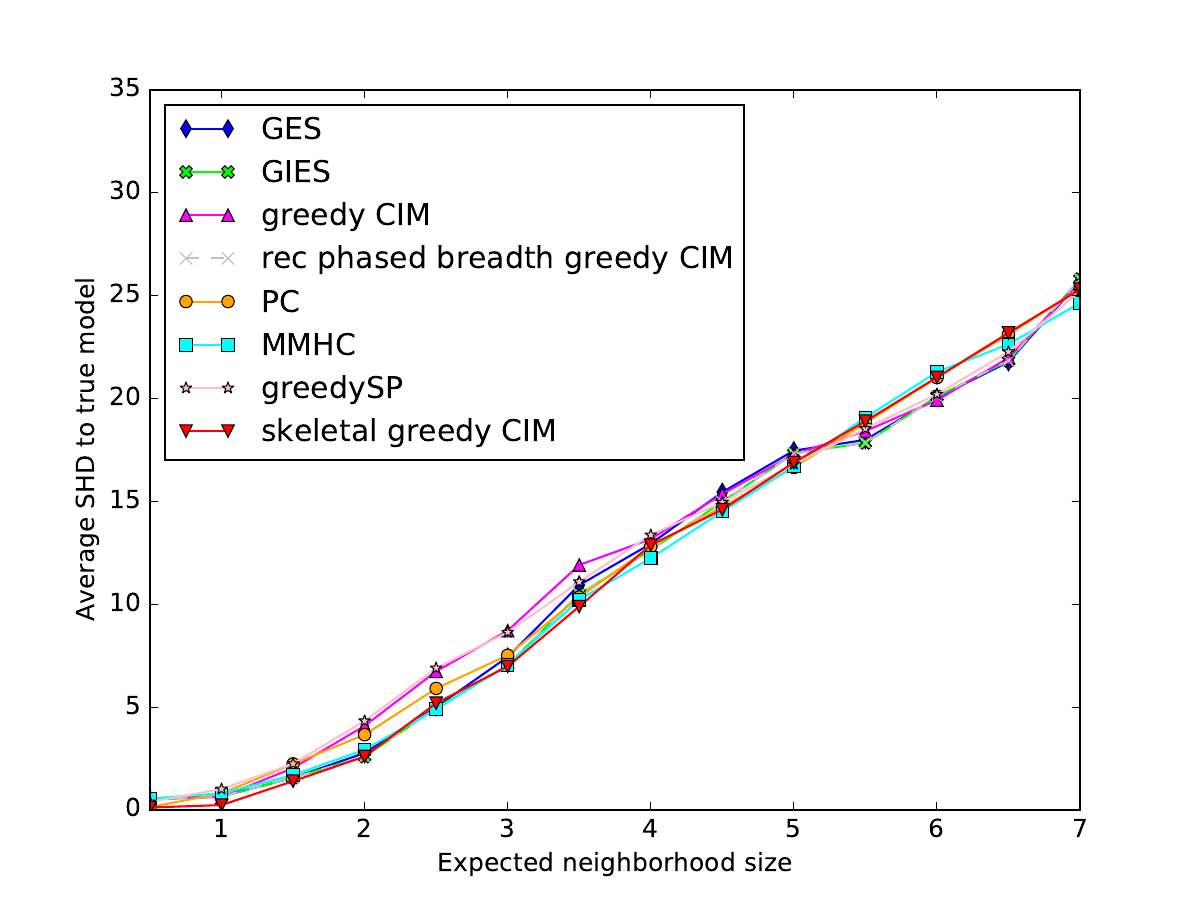}
\caption{The average SHD between the true model and the result of different algorithms.}
\label{subfig: algo shd all}
\end{subfigure}
\caption{A comparison between all algorithms discussed based on $100$ simulations.
Each model had $p=8$ nodes and the weights of the edges were sampled uniformly in $[-1, -0.25]\cup [0.25, 1]$. 
We used a sample size of $n=10,000$, and $\alpha= 0.0001$. }
\label{fig: algo shd}
\end{figure}

In \Cref{fig: algo shd} we give a complete comparison of the recovery ratios of all algorithms discussed as well as the SHD between the result for each algorithm and the true model.
Even though greedy CIM has a higher recovery ratio than the PC algorithm and MMHC, the average SHD is higher as well. 
This indicates that, while greedy CIM typically succeeds in finding the true DAG more often than these algorithms, when it fails to do so it returns a less accurate MEC than the other algorithms.
Thus, greedy CIM likely does a move early on from which it cannot move towards the optimal imset, since we do not have access to all edges of $\CIM_p$. 
(Note that, as BIC is linear over $\CIM_p$, this would never happen given a complete characterization of the edges of $\CIM_p$.)
On the other hand, skeletal greedy CIM is one of the top performers in regards to average SHD. 
As opposed to skeletal greedy CIM, greedy CIM will probably improve if more edges of $\CIM_p$ are identified. 

\section{Discussion}
\label{sec: discussion}

In this paper, we have studied the characteristic imset polytope $\CIM_p$ and its faces $\CIM_G$. 
We have shown that most common moves utilized in greedy causal discovery algorithms, such as reversing or adding an edge, correspond to edges of $\CIM_p$.
Utilizing this, we introduced skeletal greedy CIM (\Cref{alg: ske greedy cim}) and greedy CIM (\Cref{alg: greedy cim}).
These algorithms are greedy depth-first search edge-walks over the  $\CIM_G$ and $\CIM_p$ polytopes, respectively.
Skeletal greedy CIM is a hybrid algorithm that first does CI tests to learn a skeleton $G$, and then passes to a restricted edge-walk over $\CIM_G$, attempting to maximize the BIC by walking along edges labeled by turn pairs or edge pairs.
Greedy CIM performs a similar restricted edge-walk over $\CIM_p$.
Both algorithms could likewise be implemented using any score-equivalent and decomposable score function.
We showed that (recurrent phased breadth-first) greedy CIM is a geometric generalization of GES and GIES in the case of purely observational data. 
Consequently, GES and GIES admit a geometric interpretation as edge-walks along a convex polytope.
It further follows that MMHC has a similar interpretation. 
As greedy SP already has such an interpretation in terms of the DAG associahedron \cite{MUW18} it follows that all greedy algorithms discussed in this paper have a geometric interpretation as an edge-walk along a convex polytope.
In this sense, we have observed that greedy causal discovery is geometric. 

An implementation of skeletal greedy CIM and greedy CIM is available at \cite{github}.
Given data drawn from a joint distribution on 8 variables, these implementations return a graph in approximately $1$ and 10 seconds on average, respectively. 
We believe that a more efficient implementation is possible, but we leave that for future work.

Skeletal greedy CIM was shown to outperform the other hybrid algorithms such as MMHC and greedy SP on simulated Gaussian data. 
The main difference between these algorithms is that skeletal greedy CIM relies on CI tests to determine the skeleton, while MMHC only utilizes the CI tests to restrict the set of possible skeletons.
Thus it is probable that turn pairs capture many edges of $\CIM_G$, while turn and edge pairs capture relatively few edges of $\CIM_p$. 
So while skeletal greedy CIM appears to be a near optimal hybrid algorithm given its constraint-based bounds, identifying more edges of $\CIM_p$ to extend the moves used by MMHC between skeleta could lead to an algorithm capable of outperforming both skeletal greedy CIM and MMHC.  
Given that one can use \texttt{polymake} \cite{AGHJLPR17, GJ00} to compute all edges of $\CIM_4$, a natural first step would be to try to generalize some of these edges not captured by edge pairs or turn pairs to higher values of $p$.

Finally, recall that GIES first adds in edges without considering the deletion of edges, then deletes edges without considering the addition of edges, then reverses edges, and then cycles through each of these phases.
GIES also does a breadth-first search.
Thus, we believe that the depth-first nature of greedy CIM induces a preference on the edges which is avoided by GIES via a breadth-first search.
A recurrent phased breadth-first version of greedy CIM was implemented and performed identically, in terms of accuracy, with GIES in our simulations. 
A natural follow-up question is then: how often, if ever, does recurrent phased breadth-first search greedy CIM utilize the extra moves to which it has access?
Presently, what we can surmise is that finding and implementing more edges of the $\CIM_p$ polytope could lead to even better greedy causal discovery algorithms than the current front-runners (GIES and recurrent phased breadth-first greedy CIM).

\section{Acknowledgements}
All three authors were partially supported by the Wallenberg AI, Autonomous Systems and Software Program (WASP) funded by the Knut and Alice Wallenberg Foundation.
Svante Linusson was partially supported by Grant (No. 2018-05218) from Vetenskapsr\aa{}det (The Swedish Research Council).
Liam Solus was partially supported by Starting Grant (No. 2019-05195) from Vetenskapsr\aa{}det (The Swedish Research Council).
The authors thank an anonymous reviewer for helpful suggestions that greatly improved the presentation of the paper.

\bibliographystyle{siamplain}
\bibliography{references}

\appendix

\section{Proofs of Theorems in \Cref{sec: edges} }
\label{sec: proofs edges}


\subsection*{Proof of \Cref{prop: edge turn}}
We have the following equality
\[
c_{\GG_{i\leftarrow j}}=c_\GG+\sum_{S\in \mathcal{A}^+}e_S-\sum_{S\in \mathcal{A}^-}e_S
\]
for some $\mathcal{A}^+$  and $\mathcal{A}^-$.
We begin by giving a possible description of $\mathcal{A}^+$ and $\mathcal{A}^-$.
If we have a set $S$ such that $\{i,j\}\not\subseteq S$, then the  graphs induced by $\GG$ and $\GG_{i\leftarrow j}$ on $S$ are identical and we can assume that no such $S$ is in either $\mathcal{A}^+$ or $\mathcal{A}^-$.
We only changed the edge $i\to j$.
So for any set $S$, the only node that could have become the child of every other node in $S$ upon reversing $i\to j$ is $i$.
Taking this as a definition of $\mathcal{A}^+$ we get that $\mathcal{A}^+$  is all sets $S$ such that $\{i, j\}\subseteq S\subseteq \pa_{\GG_{i\leftarrow j}}(i)\cup \{i\}$ and $\{i, j\}\subseteq S\not\subseteq \pa_{\GG}(i)\cup \{i\}$.
That gives us $\mathcal{A}^+=\left\{ S\colon \{i,j\}\subseteq S\subseteq \pa_\GG(i)\cup\{i,j\}\right\}=\left\{ S\cup \{i,j\}\colon S\subseteq \pa_\GG(i)\right\}$.
Similar reasoning gives us $\mathcal{A}^-=\left\{ S\cup \{i,j\}\colon S\subseteq \pa_\GG(j)\right\}$.
Note that $\mathcal{A}^+\cap\mathcal{A}^-=\{S\cup\{i,j\}\colon S\subseteq \pa_\GG(i)\cap\pa_\GG(j)\}$.

Let $S_i=\pa_\GG(i)$ and let $S_j=\pa_\GG(j)\backslash\{i\}$. 
We will now check the conditions in \Cref{def: turn pair} with respect to $(i,j,S_i,S_j)$.
$\GG$ and $\GG_{i\leftarrow j}$ have the same skeleton, say $G$. 
Conditions (1)-(3) are direct from the definition of characteristic imset as $i$ is the child of every node in $S_i=\pa_\GG(i)$, and similarly with $j$.
If $S_i\subseteq \ne_G(j)$ we have $S_i\subseteq \pa_\GG(j)$, indeed otherwise we would have $k\in S_i=\pa_\GG(i)$ such that $k\in \ch_\GG(j)$. 
This gives us the edges $i\to j\to k\to i$ in $\GG$, a contradiction as $\GG$ is a DAG. 

{\bf Case I, $S_i\subseteq \ne_G(j)$ and $S_j\subseteq \ne_G(i)$:}
We have $i\notin S_i$. 
As argued above, if $S_i\subseteq \ne_G(j)$ and $S_j\subseteq \ne_G(i)$ we get $S_i\subseteq \pa_\GG(j)\setminus\{i\}=S_j\subseteq \pa_\GG(i)=S_i$. 
In particular $\pa_\GG(j)\setminus\{i\}= \pa_\GG(i)$. 
Thus $\GG$ and $\GG_{i\leftarrow j}$ are Markov equivalent. 
This was first proved by Chickering in \cite{C95}. 
From the viewpoint of imsets, we get $\mathcal{A}^+=\mathcal{A}^-$ and thus $c_\GG=c_{\GG_{i\leftarrow j}}$.

{\bf Case II, $S_i\not\subseteq \ne_G(j)$ or $S_j\not\subseteq \ne_G(i)$:}
Condition (4) in \Cref{def: turn pair} holds by assumption.
Thus what is left is to check that $\SSS^+=\mathcal{A}^+\setminus\mathcal{A}^-$ and $\SSS^-=\mathcal{A}^-\setminus\mathcal{A}^+$.
Then by our above reasoning we get
\begin{align*}
\mathcal{A}^+\setminus\mathcal{A}^-&=\left\{ S\cup \{i,j\}\colon S\subseteq \pa_\GG(i), S\not\subseteq \pa_\GG(j)\right\}\\
&=\left\{ S\cup\{i,j\}\colon S\subseteq S_i, S\not\subseteq\ne_G(j)\right\}=\SSS^+.
\end{align*}
Similar reasoning gives us $\SSS^-=\mathcal{A}^-\setminus\mathcal{A}^+$.
\hfill\qed

For the following proofs we will use the following well-known fact.
\begin{lemma}
\label{lem: hamilton dist one edge}
Let $P$ be a 0/1-polytope. 
If $u$ and $v$ are two vertices of $P$ such that $u$ and $v$ differ by a single value. 
Then $\conv(u,v)$ is an edge of $P$. 
\end{lemma}

\subsection*{Proof of \Cref{thm: turn pair}}
By definition we have $c_\GG(\{k, i\})=1$ for all $k\in S_i$, thus $S_i\subseteq\ne_G(i)$ and similar for $S_j$. 
Note that this implies that $\SSS^+$ and $\SSS^-$ are disjoint.
If $S_i=S_j$ we have that $S_i=S_j\subseteq \ne_G(j)$, and vice versa, thus this is not a turn pair. 
By symmetry in the definition we get two cases.

{\bf Case I, $S_j\subsetneq S_i$:}
If $|S_i|=|\{k\}|=1$ we get that $c_\HH=c_\GG+e_{\{i,j, k\}}$, and thus this follows by \cref{lem: hamilton dist one edge}.
To prove the claim when $|S_i|\geq 2$, it suffices to find a cost vector $w\in\RR^{2^p-p-1}$ such that $w^Tx$ is maximized at exactly $c_\GG$ and $c_\HH$ over the vertices of $\CIM_G$. 
Since $c_\GG(\{i,k\})=1$ for all $k\in S_j$ we have $S_i\subseteq\ne_G(i)$. 
Thus $S_j\subseteq S_i\subseteq\ne_G(i)$ and we get that $\SSS^- =\emptyset$, by definition of $\SSS^-$. 
Moreover, by (4) in \cref{def: turn pair}, $S_i\not\subseteq\ne_G(j)$. 
Let $m \coloneqq |\SSS^+|$ and define the cost vector $w$ such that for $S\subseteq [p]$, with $|S|\geq 2$, $w$ satisfies
\[
w(S) = 
\begin{cases}
2	&	\text{ if $c_\GG(S) =1$}\\
1 & \text{ if $S = S_i\cup\{i,j\}$},		\\
\frac{-1}{m-1}	&	\text{ if $S\in\SSS^+\setminus\{S_i\cup\{i,j\}\}$},	\\
-2	&	\text{otherwise}.	\\
\end{cases}
\]
Notice that since $|S_i|\geq2$ we have $m\geq 2$ so this is indeed well defined.
Then we have $w^Tc_\GG = w^Tc_\HH$ since 
\begin{equation*}
\begin{split}
w^Tc_\HH
&= w^T\left(c_\GG+\sum_{S\in\SSS^+}e_S-\sum_{S\in\SSS^-}e_S\right)	\\
&= w^Tc_\GG+w(S_i\cup\{i,j\})(1)+\sum_{S\in\SSS^+\setminus\{S_i\cup\{i,j\}\}}w(S)= w^Tc_\GG.	\\
\end{split}
\end{equation*}
It then remains to check that $w^Tc_\DD<w^Tc_\GG$ for any DAG  $\DD$ with skeleton $G$ and $\DD$ not Markov equivalent to $\GG$ or $\HH$.  

Let us denote $\mathcal{A}^+\coloneqq \{S\colon w(S)=2\}$ and $\mathcal{A}^-\coloneqq\{S\colon w(S)=-2\}$. 
For all $0/1$-vectors $v$ we have 
\begin{align*}
w^Tv&=w^T\sum_{S\in\mathcal{A}^+\colon v(S)=1}e_S+w^T\sum_{S\in\mathcal{A}^-\colon v(S)=1}e_S+w^T\sum_{S\in\SSS^+\colon v(S)=1}e_S\\
&=2\left|\left\{S\in\mathcal{A}^+\colon v(S)=1\right\}\right|-2\left|\left\{S\in\mathcal{A}^-\colon v(S)=1\right\}\right|+w^T\sum_{S\in\SSS^+\colon v(S)=1}e_S.
\end{align*}
Noting that $c_\GG(S)=c_\HH(S)=1$ for all $S\in\mathcal{A}^+$, $c_\GG(S)=c_\HH(S)=0$ for all $S\in\mathcal{A}^-$ and that 
$
-1\leq w^T\sum_{S\in\SSS^+\colon v(S)=1}e_S\leq 1
$
we immediately get that $w^Tv<w^Tc_\GG$ whenever we have that $\left\{S\in\mathcal{A}^+\colon v(S)=1\right\}\neq \mathcal{A}^+$ or $\left\{S\in\mathcal{A}^-\colon v(S)=0\right\}\neq \mathcal{A}^-$.
Then as $\left\{S\subseteq[p]\colon |S|\geq 2\right\}=\mathcal{A}^+\cup\mathcal{A}^-\cup\SSS^+$ we can assume that $c_\DD(S)=c_\GG(S)$ whenever $S\notin\SSS^+$.
In particular $\DD$ must have the same skeleton as $\GG$ and $\HH$. 

Since $\DD$ was assumed to not be Markov equivalent to $\GG$ we have the following cases:
\begin{enumerate}[label=(\arabic*)]
\item{$c_\DD(S_i\cup\{i,j\})=0$ and for some set $S\in\SSS^+\backslash \{S_i\cup\{i,j\}\}$ we have $c_\DD(S)=1$, or}
\item{$c_\DD(S_i\cup\{i,j\})=1$.}
\end{enumerate}

In case (1) it follows immediately that $w^Tc_\DD\leq w^Tc_\GG+\frac{-1}{m-1}<w^Tc_\GG$.

As for case (2), by definition of the characteristic imset we have a node $n$ such that $x\to n$ in $\DD$ for all $x\in \left(S_i\cup\{i,j\}\right)\backslash\{n\}$. 
If $n= j$ we get $S_i\subseteq \ne_G(j)$, but this cannot happen by (4) in \cref{def: turn pair}. 
If $n=i$ we get that $c_\DD(S)=1$ for all $S\in \SSS^+$, and thus $\DD$ is Markov equivalent to $\HH$. 
Thus the only case left is that $n\in S_i$.

As $S_i\not\subseteq\ne_G(j)$ we have $S_i \cup\{i,j\}\in\SSS^+$.
Then, as $j\to n$ in $\DD$, there must exist a node $k\notin\{i,j,n\}$ such that $k$ is not a neighbour of $j$ in $G$.
Since $\{j,k\}\subseteq S_i\cup\{i,j\}\subseteq \pa_\DD(n)\cup \{n\}$ we get $c_\DD(\{j,n,k\})=1$.
As $i\notin\{j,n,k\}$, $\{j,n,k\}\notin \SSS^+$.
Thus we must have that $1=c_\DD(\{j,n,k\})=c_\GG(\{j,n,k\})=c_\HH(\{j,n,k\})$. 
That is $\{j,n,k\}$ is a v-structure in $\DD$, $\GG$ and $\HH$. 
We have that $c_\GG(\{i,j,k\})=0$ since $\{i,j,k\}\in\SSS^+$.
Thus, since $\GG$ is acyclic, it follows that $i\to n$ in $\GG$ as well.
In the terminology used in \cite{AMP97}, $i\to n$ will be strongly protected in $\GG$.
Hence $n$ is a child of $i$, $j$ and $k$ in $\GG$, so $c_\GG(\{i,j,n,k\})=1$.
But $\{i,j,n,k\}\in\SSS^+$, a contradiction.

{\bf Case II, $S_i\not\subseteq S_j$ and $S_j\not\subseteq S_i$:}
Here we will use a different cost vector. 
Let $m^+\coloneqq |\SSS^+|$ and $m^-\coloneqq|\SSS^-|$. 
If $m^+, m^-\geq2$ define 
\[
w(S) =
\begin{cases}
5	&	\text{ if $c_\GG(S) = c_\HH(S)=1$},	\\
2	&	\text{ if $S=S_i\cup\{i,j\}$ or $S=S_j\cup\{i,j\}$}\\
\frac{-1}{m^+-1} 	&	\text{ if $S \in \SSS^+\setminus \{S_i\cup\{i,j\}\}$}, 	\\
\frac{-1}{m^--1} 	&	\text{ if $S \in \SSS^-\setminus \{S_j\cup\{i,j\}\}$}, 	\\
-5	&	\text{ if $c_\GG(S) = c_\HH(S)=0$},	\\
\end{cases}
\]

If $|\SSS^+|=1$ we have that $\SSS^+ =\{S_i\cup\{i,j\}\}$, and thus we let $w(S_i\cup\{i,j\})=1$. 
Likewise, if $|\SSS^-|=1$ we have that  $\SSS^- =\{S_j\cup\{i,j\}\}$, and we let $w(S_j\cup\{i,j\})=1$. 
Otherwise let $w$ be as above.
Thus, by definition of $w$, we have 
$
\sum_{S\in\SSS^+}w(S)=\sum_{S\in\SSS^-}w(S)=1. 
$
To see $w^Tc_\HH=w^Tc_\GG$, note that
\begin{align*}
w^Tc_\HH-w^Tc_\GG=& w^T\left( c_\GG +\sum_{S\in \SSS^+}e_S - \sum_{S\in \SSS^-}e_S\right)-w^Tc_\GG\\
=&\sum_{S\in \SSS^+\setminus \SSS^-}w(S) - \sum_{S\in S \in \SSS^-\setminus \SSS^+}w(S)=0.
\end{align*}
So left to show is that for any DAG $\DD$ with skeleton $G$ we have $w^Tc_\DD < w^Tc_\GG$ if $c_\DD$ is neither $c_\GG$ or $c_\HH$.

As in case I we let $\mathcal{A}^+\coloneqq \left\{S\colon w(S)=5\right\}=\{S\colon c_\GG(S)=c_\HH(S)=1\}$ and  $\mathcal{A}^-\coloneqq \left\{S\colon w(S)=-5\right\}=\{S\colon c_\GG(S)=c_\HH(S)=0\}$.
As in case I we have for any $0/1$ vector $v$  
\begin{align*}
w^Tv=\ &
5\left|\left\{S\in\mathcal{A}^+\colon v(S)=1\right\}\right|-5\left|\left\{S\in\mathcal{A}^-\colon v(S)=1\right\}\right|\\
&+w^T\sum_{S\in\SSS^+\colon v(S)=1}e_S+w^T\sum_{S\in\SSS^-\colon v(S)=1}e_S.\\
\end{align*}
We also have that
$
-1\leq w^T\sum_{S\in\SSS^+\colon v(S)=1}e_S\leq 2
$
and
$
-1\leq w^T\sum_{S\in\SSS^-\colon v(S)=1}e_S\leq 2.
$
We immediately get that $w^Tv<w^Tc_\GG$ whenever we have that $\left\{S\in\mathcal{A}^+\colon v(S)=1\right\}\neq \mathcal{A}^+$ or $\left\{S\in\mathcal{A}^-\colon v(S)=0\right\}\neq \mathcal{A}^-$. 
Thus we can assume that $c_\DD(S)=c_\GG(S)$ whenever $c_\GG(S)=c_\HH(S)$.

If $c_\DD(S_i\cup\{i,j\})=c_\DD(S_j\cup\{i,j\})=0$ then it follows that $w^Tc_\DD\leq 5|\mathcal{A}^+|<5|\mathcal{A}^+|+1=c_\GG$. 
Thus for $w^Tc_\DD\geq w^Tc_\GG$ to be true we must have $c_\DD(S_i\cup\{i,j\})=1$ or $c_\DD(S_j\cup\{i,j\})=1$. 
By symmetry we can assume $c_\DD(S_i\cup\{i,j\})=1$.

Thus there exists $n_i\in S_i$ such that $S_i\cup\{i,j\}\subseteq\pa_\DD(n_i)\cup\{n_i\}$.
We cannot have $n_i=j$ as that would give us $S_i\subseteq\ne_G(j)$, and by the same reasoning there must exist a node $k_i\in S_i\setminus\ne_G(j)$. 
Then we have two cases $n_i\neq i$ and $n_i=i$.

If $n_i\neq i$ we have that $c_\DD(\{n_i,k_i,j\})=1$. 
As $\{n_i,k_i,j\}\notin\SSS^+\cup\SSS^-$ we get $c_\DD(\{n_i,k_i,j\})=c_\GG(\{n_i,k_i,j\})=c_\HH(\{n_i,k_i,j\})=1$.
Then by acyclicity we get $c_\GG(\{n_i,k_i,i,j\})=1$.
But as $\{n_i,k_i,i,j\}\in\SSS^+\backslash\SSS^-$ we get $c_\GG(\{n_i,k_i,i,j\})=0$, a contradiction. 

Thus $n_i=i$.
Then, by definition, it follows that $c_\DD(S)=1$ for all $S\in\SSS^+$. 
If $c_\DD(S_j\cup\{i,j\})=1$ we can in the same way argue that the corresponding $n_j=j$ and thus that $c_\DD(S)=1$ for all $S\in\SSS^-$.
More specifically we get that we have the following two graphs induced in $\DD$, $i \to j\leftarrow k_j$ and $k_i\to i\leftarrow j$. 
A contradiction, thus if $c_\DD(S_i\cup\{i,j\})=1$ we have $c_\DD(S_j\cup\{i,j\})=0$. 

In conclusion, we assumed that $w^Tc_\DD\geq w^Tc_\GG$ and deduced that we cannot have both $c_\DD(S_i\cup\{i,j\})=1$ and $c_\DD(S_j\cup\{i,j\})=1$. With that assumption it also followed that if $c_\DD(S_i\cup\{i,j\})=1$ then $c_\DD(S)=c_\GG(S)$ for all $S$. By symmetry, if $c_\DD(S_j\cup\{i,j\})=1$ then $c_\DD(S)=c_\HH(S)$ for all $S$. The result follows.
\hfill\qed


\subsection*{Proof of \Cref{prop: edge addition}}
We begin to characterize all sets $S$ such that $c_\GG(S)\neq c_{\GG_{+i\leftarrow j}}(S)$. 
For any $S\subseteq [p]$ and $k\neq i$ we have that $k\in S\subseteq \pa_\GG(k)\cup\{k\}$ if and only if $k\in S\subseteq \pa_{\GG_{+i\leftarrow j}}(k)\cup\{k\}$. 
This is because $\pa_\GG(k)=\pa_{\GG_{+i\leftarrow j}}(k)$ for all such $k$. 
As the value of $c_\GG(S)$ and $c_{\GG_{+i\leftarrow j}}(S)$ is determined by this property the only case where we can have $c_{\GG}(S)\neq c_{\GG_{+i\leftarrow j}}(S)$ is for sets such that $i\in S\not\subseteq\pa_\GG(i)\cup\{i\}$ or $i\in S\subseteq\pa_{\GG_{+i\leftarrow j}}(i)\cup\{i\}$.

Moreover, for any $S$ such that $\{i,j\}\centernot\subseteq S$ we have that the induced subgraphs of $\GG$ and $\GG_{+i\leftarrow j}$ are identical. 
Thus $c_{\GG}(S)=c_{\GG_{+i\leftarrow j}}(S)$ for all such $S$. 
This together with the fact that $\pa_\GG(i)\cup\{j\}=\pa_{\GG_{+i\leftarrow j}}(i)$ tells us that the only sets of interest are $\{i,j\}\subseteq S\subseteq \pa_\GG(i)\cup\{i,j\}$.

We claim that $c_{\GG_{+i\leftarrow j}}(S)=1$ and $c_\GG(S)=0$ for all $S$ such that $\{i,j\}\subseteq S\subseteq \pa_\GG(i)\cup\{i,j\}$, making this an edge pair with respect to $(i,j,S^\ast)$ where $S^\ast = \pa_\GG(i)$.
It follows that $c_{\GG_{+i\leftarrow j}}(S)=1$ for all such $S$ since $i\in S\subseteq  \pa_\GG(i)\cup\{i,j\} =  \pa_{\GG_{+i\leftarrow j}}(i)\cup\{i\}$.
Suppose $S$ is such that $\{i,j\}\subseteq S\subseteq \pa_\GG(i)\cup\{i,j\}$.
Any $k\in S\setminus\{i,j\}$ must be a parent of $i$ in $\GG$, since we cannot have $i\in S\subseteq \pa_\GG(k)\cup\{k\}$.
As $i$ and $j$ are not adjacent neither can be the parent of the other. 
Hence no node in $S$ can be the parent of all other nodes in $S$, and it follows that $c_\GG(S) = 0$. 
Condition (1) in \Cref{def: edge pair} follows since $i$ was not a neighbor of $j$ in $\GG$, and condition (2) follows since we choose $S^\ast$ to be $\pa_\GG(i)$.
\hfill\qed

\subsection*{Proof of \Cref{thm: edge pair}}
If $|S^\ast |=0$ we get $|\SSS_{+i\leftarrow j}|=1$, thus this follows by \Cref{lem: hamilton dist one edge}.
Hence we can assume that $|S^\ast |>0$. 
We partition the elements in $S^\ast$ based on if they are adjacent to $j$ in $\GG$ or not.
So let $X=\{x\in S^\ast \colon c_\GG(\{x,j\})=0\}$ and $Y=\{y\in S^\ast \colon c_\GG(\{y,j\})=1\}$. 
Define $s\coloneqq|X|$ and $t\coloneqq|Y|$. 
We treat the cases when $X\neq\emptyset$ and $X=\emptyset$ separately.
\\
\noindent
{\bf Case I, $X\neq \emptyset$:}
Let $M\coloneqq 2^{s+t}-(s+t+1)$ and notice that $\left|\{S\in \SSS_{+i\leftarrow j}, |S|\ge 4\}\right|=M$.
If $M=0$ we get $s=1$ and $t=0$ as $|X|=s$.
In this case $c_\GG$ and $c_\HH$ only differ in the coordinates $\{i, j\}$ and $\{i,j\}\cup X$. 
We claim that $c_\GG + e_{\{i,j\}\cup X}$ is not a valid imset as $\{i,j\}\cup X$ is not connected in the skeleton of $\GG$. 
Hence $c_\GG(S)$, $c_\HH(S)$ and at most one more vertex in $\CIM_p$ form a face of $\CIM_p$. 
It follows that $\conv(c_\GG, c_\HH)$ is an edge in this case.

If $M>0$ we can define the following objective function $w$ to prove that $\conv(c_\GG, c_\HH)$ is an edge of $\CIM_p$:
\[
w(S) = 
\begin{cases}
t+2	&	\text{ if $c_\GG(S) =1$}\\
-1	&	\text{ if $S=\{i,j\}$},	\\
-1    & \text{ if $S =\{i,j,y\}$, some $y\in Y$},		\\
\frac{1}{s}(t+\frac12)    & \text{ if $S =\{i,j,x\}$, some $x\in X$},		\\
\frac{1}{2M}	&	\text{ if $S\in \SSS_{+i\leftarrow j}, |S|\ge 4$}, \\
-(t+2)	&	\text{otherwise}.	\\
\end{cases}
\]

The negative weights for $S\in \SSS_{+i\leftarrow j} $ sum to $-(t+1)$ and the positive to $t+1$. 
Since the imsets differ exactly on $\SSS_{+i\leftarrow j}$, for which $w$ sum to 0, we get $w^T c_\HH=w ^T c_\GG$. 
Assume we have a DAG $\DD$ such that $w^T c_\DD\ge w ^T c_\GG$.
Then it must be that $c_\DD(S)=1$ if $c_\GG(S)=1$ and $c_\DD(S)=0$ if $c_\HH(S)=0$.
If $c_\DD(S)=0$ for all $S\in \SSS_{+i\leftarrow j}$ then $c_\DD=c_\GG$, so we can assume that is not the case.
Such a DAG $\DD$ must thus pick up some of the positive weights in $\SSS_{+i\leftarrow j}$. 
There are two possibilities to consider. 
First, if $c_\DD(\{i,j,x\})=1$ for some $x\in X$, then, by definition of $c_\DD$, $\DD$ must have v-structure $x\to i \leftarrow j$, since we know there is no edge between $x$ and $j$. 
Therefore we must have $c_\DD(\{i,j\})=1$, and it follows that $c_\DD(\{i,j,y\})=1$, for all $y\in Y$, since $y$ is adjacent to both $i$ and $j$.
Thus $w^T c_\DD$ picks up all the negative weights in $\SSS_{+i\leftarrow j}$.
To then get $w^T c_\DD\geq w ^T c_\GG$, we must have $c_\DD(S)=c_\HH(S)$ for all $S$.
Therefore, $\DD$ is Markov equivalent to $\HH$ by \Cref{lem: studeny}.

Second, if $c_\DD(\{i,j,x\})=0$ for all $x\in X$, but $c_\DD(S)=1$, for some $S\in \SSS_{+i\leftarrow j}, |S|\ge 4$, 
then by definition there exists $k\in S$ with $S\subseteq \pa_\DD(k)\cup \{k\}$.  
If $k\in\{i,j\}$ we immediately get $c_\DD(\{i,j\})=1$. 
Otherwise we have $c_\DD(\{i,j,k\})=1$, and since there is no edge between $j$ and elements in $X$ we know that $k\in Y$.
In either case $w^T c_\DD$ picks up a $-1$. 
The sum of the positive weights $w(S)$ for $S\in \SSS_{+i\leftarrow j}, |S|\ge 4$ is only $\sfrac{1}{2}$ and we cannot have $w^T c_\DD\ge w^T c_\GG$.

\noindent
{\bf Case II, $X= \emptyset$:} 
If $M=0$, either $s+t=1$ or $s+t=0$.  The latter implies $S^\ast = \emptyset$, which is dealt with above.  For the former, we get $s=0$ and $t=1$. 
We claim that $c_\GG+e_{\{i,j\}}$ is not a valid characteristic imset for any DAG, since $\{i,j\}\cup Y$ is complete in the skeleton of $\HH$. Similar to Case I it follows $\conv(c_\GG,c_\HH)$ is an edge.

If $M>0$ we now use the following objective function $w$ in order to prove that $\conv(c_\GG, c_\HH)$ is an edge of $\CIM_p$:
\[
w(S) = 
\begin{cases}
t+1	&	\text{ if $c_\GG(S) =1$}\\
t-\frac12	&	\text{ if $S=\{i,j\}$},	\\
-1    & \text{ if $S =\{i,j,y\}$, some $y\in Y$},		\\
\frac{1}{2M}	&	\text{ if $S\in \SSS_{+i\leftarrow j}, |S|\ge 4$}, \\
-(t+1)	&	\text{otherwise}.	\\
\end{cases}
\]
Here $s=0$, so $M=2^t-t-1$.
The reasoning is very similar to Case I. The negative weights for $S\in \SSS_{+i\leftarrow j} $ sum to $-t$ and the positive to $t$. Thus, $w^T c_\HH=w^T c_\GG$, and again
if another DAG $\DD$ were to have $w^T c_\DD\ge w^T c_\GG$, then it must have $c_\DD(S)=1$ if $c_\GG(S)=1$ and $c_\DD(S)=0$ if $c_\HH(S)=0$.
There are two possibilities to consider. 
First, if $c_\DD(\{i,j\})=1$, then $\DD$ has  triangles on every $\{i,j,y\}$ and therefore $c_\DD(\{i,j,y\})=1$, for all $y\in Y$. Thus $\DD$ picks up all the $-t$ negative weights and the only possibility is $c_\DD=c_\HH$.
The second possibility is that $c_\DD(\{i,j\})=0$ but $c_\DD(S)=1$, for some $S\in \SSS_{+i\leftarrow j}, |S|\ge 4$, then by definition there exists $k\in S$ with $S\subseteq \pa_\DD(k)\cup \{k\}$.  
As $i$ and $j$ are not adjacent we get $k\centernot\in \{i,j\}$.
This implies that $c_\DD(\{i,j,k\})=1$, for $k\in Y$, which gives a $-1$ in $w^T c_\DD$.
The sum of the positive weights $w(S)$ for $S\in \SSS_{+i\leftarrow j}, |S|\ge 4$ is $\sfrac12$ and thus we cannot have $w^T c_\DD\ge w^T c_\GG$.
\hfill\qed

\section{The Turn Phase and the Edge Phase Algorithms}
\label{sec: algorithms}
Here we present the pseudocode for the edge phase and turn phase used in \Cref{alg: ske greedy cim} and \Cref{alg: greedy cim}.
The edge phase and turn phase algorithms are presented in \Cref{alg: edge phase} and \Cref{alg: turn phase}, respectively.

\begin{algorithm}[t!]
\caption{Edge phase}
\label{alg: edge phase}
\raggedright
\textbf{Input:}{ An imset $c_\GG$ corresponding to a DAG $\GG$. Data ${\bf D}$.}\\
\textbf{Output:}{ A characteristic imset $c_\GG$ where $\GG$ is a DAG.}
\begin{algorithmic}
\State{Let $G$ be the skeleton of $\GG$}
\State{$check \leftarrow \texttt{true}$}
\While{$check$}
    \State{$check \leftarrow \texttt{false}$}
    \For{$i,j\in [p]$}
        \For{$S^\ast\subseteq \ne_G(i)$}
            \If{We have a DAG $\HH$ such that $\{\GG, \HH\}$ is an edge pair with respect to $(i,j,S^\ast)$}
                \If{$\BIC(\HH,{\bf D}) > \BIC(\GG,{\bf D})$}
                    \State{$c_\GG\gets c_\HH$}
                    \State{Let $G$ be the skeleton of $\GG$}
                    \State{$check \leftarrow \texttt{true}$}
                    \State{{\bf break}}
                \EndIf
            \EndIf
        \EndFor
    \EndFor
\EndWhile\\
\Return{$c_\DD$}
\end{algorithmic}
\end{algorithm}

\begin{algorithm}[t!]
\caption{Turn phase}
\label{alg: turn phase}
\raggedright
\textbf{Input:}{ An imset $c_\GG$ corresponding to a DAG $\GG$. Data ${\bf D}$.}\\
\textbf{Output:}{ A characteristic imset $c_\GG$ where $\GG$ is a DAG.}
\begin{algorithmic}
\State{$c_\DD\leftarrow c_\GG$}
\State{Let $G$ be the skeleton of $\GG$}
\State{$check \leftarrow \texttt{true}$}
\While{$check$}
    \State{$check \leftarrow \texttt{false}$}
    \For{$i,j\in [p]$}
        \For{$S_i\subseteq \ne_G(i)$ and $S_j\subseteq \ne_G(j)$}
            \If{We have a DAG $\HH$ such that $\{\DD, \HH\}$ is an turn pair with respect to $(i,j,S_i, S_j)$}
                \If{$\BIC(\HH,{\bf D}) > \BIC(\DD,{\bf D})$}
                    \State{$c_\DD\gets c_\HH$}
                    \State{$check \leftarrow \texttt{true}$}
                \EndIf
            \EndIf
        \EndFor
    \EndFor
\EndWhile\\
\Return{$c_\GG$}
\end{algorithmic}
\end{algorithm}

\end{document}